\documentclass[a4paper,11pt]{amsart}

\usepackage[T1]{fontenc}
\usepackage[utf8]{inputenc}
\usepackage{lmodern}
\usepackage{microtype}

\usepackage{fullpage}

\usepackage{amssymb}
\usepackage{physics} 
\usepackage{tensor}
\usepackage{stmaryrd} 

\usepackage[backref]{hyperref} 

\usepackage[nobysame,alphabetic,initials,msc-links]{amsrefs}

\DefineSimpleKey{bib}{how}
\renewcommand{\PrintDOI}[1]{%
  \href{http://dx.doi.org/#1}{{\tt DOI:#1}}%
}
\renewcommand{\eprint}[1]{#1}
\BibSpec{book}{%
    +{}  {\PrintPrimary}                {transition}
    +{.} { \PrintDate}                  {date}
    +{.} { \textit}                     {title}
    +{.} { }                            {part}
    +{:} { \textit}                     {subtitle}
    +{,} { \PrintEdition}               {edition}
    +{}  { \PrintEditorsB}              {editor}
    +{,} { \PrintTranslatorsC}          {translator}
    +{,} { \PrintContributions}         {contribution}
    +{,} { }                            {series}
    +{,} { \voltext}                    {volume}
    +{,} { }                            {publisher}
    +{,} { }                            {organization}
    +{,} { }                            {address}
    +{,} { }                            {status}
    +{,} { \PrintDOI}                   {doi}
    +{,} { \PrintISBNs}                 {isbn}
    +{}  { \parenthesize}               {language}
    +{}  { \PrintTranslation}           {translation}
    +{;} { \PrintReprint}               {reprint}
    +{.} { }                            {note}
    +{.} {}                             {transition}
    +{}  {\SentenceSpace \PrintReviews} {review}
}
\BibSpec{article}{%
    +{}  {\PrintAuthors}                {author}
    +{,} { \textit}                     {title}
    +{.} { }                            {part}
    +{:} { \textit}                     {subtitle}
    +{,} { \PrintContributions}         {contribution}
    +{.} { \PrintPartials}              {partial}
    +{,} { }                            {journal}
    +{}  { \textbf}                     {volume}
    +{}  { \PrintDatePV}                {date}
    +{,} { \issuetext}                  {number}
    +{,} { \eprintpages}                {pages}
    +{,} { }                            {status}
    +{,} { \PrintDOI}                   {doi}
    +{,} { \eprint}        {eprint}
    +{}  { \parenthesize}               {language}
    +{}  { \PrintTranslation}           {translation}
    +{;} { \PrintReprint}               {reprint}
    +{.} { }                            {note}
    +{.} {}                             {transition}
    +{}  {\SentenceSpace \PrintReviews} {review}
}
\BibSpec{collection.article}{%
    +{}  {\PrintAuthors}                {author}
    +{,} { \textit}                     {title}
    +{.} { }                            {part}
    +{:} { \textit}                     {subtitle}
    +{,} { \PrintContributions}         {contribution}
    +{,} { \PrintConference}            {conference}
    +{}  {\PrintBook}                   {book}
    +{,} { }                            {booktitle}
    +{,} { \PrintDateB}                 {date}
    +{,} { pp.~}                        {pages}
    +{,} { }                            {publisher}
    +{,} { }                            {organization}
    +{,} { }                            {address}
    +{,} { }                            {status}
    +{,} { \PrintDOI}                   {doi}
    +{,} { \eprint}        {eprint}
    +{}  { \parenthesize}               {language}
    +{}  { \PrintTranslation}           {translation}
    +{;} { \PrintReprint}               {reprint}
    +{.} { }                            {note}
    +{.} {}                             {transition}
    +{}  {\SentenceSpace \PrintReviews} {review}
}
\BibSpec{misc}{%
  +{}{\PrintAuthors}  {author}
  +{,}{ \textit}      {title}
  +{.}{ }             {how}
  +{}{ \parenthesize} {date}
  +{,} { available at \eprint}        {eprint}
  +{,}{ available at \url}{url}
  +{,}{ }             {note}
  +{.}{}              {transition}
}

\newcommand{\R}{\mathbb R}
\newcommand{\C}{\mathbb C}

\newcommand{\s}[1]{\mathcal{#1}}
\newcommand{\bb}[1]{\mathbb{#1}}

\newcommand{\fr}[1]{\mathfrak{#1}}
\newcommand{\sm}[1]{\mathsf{#1}}
\newcommand{\rP}{\mathrm P}
\newcommand{\rT}{\mathrm T}
\newcommand{\vph}{\varphi}
\newcommand{\vth}{\vartheta}

\newcommand{\ol}[1]{\overline{#1}}
\newcommand{\ox}{\otimes}
\newcommand{\pr}{\prime}

\newcommand{\Ad}{\mathrm{Ad}}
\newcommand{\ad}{\mathrm{ad}}

\newcommand{\SU}{\mathrm{SU}}
\newcommand{\SO}{\mathrm{SO}}
\newcommand{\rE}{\mathrm{E}}
\newcommand{\rU}{\mathrm{U}}
\newcommand{\Gr}{\mathrm{Gr}}
\newcommand{\medwedge}{{\textstyle\bigwedge}}

\newcommand{\scrap}[1]{} 

\newcommand{\largewedge}{\bigwedge}

\DeclareMathOperator{\Sym}{Sym}

\DeclareMathOperator{\Lin}{Lin}

\newtheorem{thm}{Theorem}[section]
\newtheorem{prop}[thm]{Proposition}
\newtheorem{lem}[thm]{Lemma}
\newtheorem{cor}[thm]{Corollary}

\theoremstyle{definition}

\theoremstyle{remark}
\newtheorem{rmk}[thm]{Remark}
\newtheorem{ex}[thm]{Example}

\begin{document}

\title{Poisson--Lie Group structures on semidirect products}

\author{Floris Elzinga}
\address{Department of Mathematics, University of Oslo, P.O.\ box 1053, Blindern, 0316 Oslo, Norway}
\email{florise@math.uio.no}

\author{Makoto Yamashita}
\email{makotoy@math.uio.no}

\thanks{The authors were partially supported by the NFR funded project 300837 ``Quantum Symmetry''.}

\date{10.10.2022, minor revisions; v1: 28.03.2022}

\begin{abstract}
We look at the Poisson structure on the total space of the dual bundle to the Lie algebroid arising from a matched pair of Lie groups.
This dual bundle, with the natural semidirect product group structure, becomes a Poisson--Lie group as suggested by a recent work of Stachura.
Moreover, when we start from matched pairs given by the Iwasawa decomposition of simple Lie groups, the associated Lie bialgebra is coboundary.
\end{abstract}

\maketitle

\section{Introduction}

Matched pairs of subgroups have been used to produce interesting examples of Hopf algebras \cite{MR611561} in the 80's.
Its Hopf algebraic analogue, such as the Drinfeld double of a Hopf algebra, proved to be a very fruitful source of interesting objects as further elaborated upon by Majid \cite{MR1045735} and collaborators.

More recently in \citelist{\cite{MR3681683}\cite{MR3950818}}, Stachura gave a groupoid quantization of the $\kappa$-Poincaré group \cite{MR1280372}.
His model is based on a Lie groupoid arising from a matched pair of subgroups in $\SO(N+1,1)$, and its associated Lie algebroid. 
The total space of the dual bundle of a Lie algebroid can be given the structure of a Poisson manifold \cite{MR998124}, and the main result of \cite{MR3681683} amounts to identifying this Poisson structure with that of the $\kappa$-Poincaré group.

Motivated by this work, we look at the general case in the framework of matched pairs of Lie groups, and a particular case arising from the Iwasawa decomposition for real simple Lie groups.

In the general setting, suppose that $B, C \subset G$ is a matched pair of Lie groups with their Lie algebras $\fr b$, $\fr c$, and $\fr g$, meaning that $B \cap C = \{ e \}$ and $B C$ is an open subset of $G$.
The induced partial action of $C$ on $B$ defines a Lie groupoid $\s{G}_B = B C \cap C B$ with base $B$, and vice versa.
The dual bundle $E$ of the associated Lie algebroid is trivializable, and the fibers can be identified with the annihilator $\fr b^0 \subset \fr g^*$ of $\fr b$.

Our starting observation is that $B$ naturally acts on $\fr b^0$, whence $E$ has a semidirect group structure by combining this action with the linear group structure on $\fr b^0$.
Our main result (Theorem \ref{thm:PLGS}) is that the Poisson structure on $E$ is multiplicative with respect to this group structure.
Hence $E$ becomes a Poisson--Lie group.

This can be further motivated by the fact that, when $B C \subset G$ is dense, the operator algebra associated with $\s{G}_B$ represents a locally compact quantum group, namely the bicrossed product of $B$ and $C$ \cite{MR1969723}.
In this scheme, the crossed product of $C$ with respect to the induced action on the function algebra of $G / C$, identified with the function algebra of $B$, corresponds to the quantized algebra of functions on the Poisson--Lie group $E$.
However, as the action of $B$ on $C$ is not by group automorphisms, we `break' the group structure on $C$ and look at the semidirect product of $\fr b^0 \simeq \fr c^*$ by $B$.

As an example, when we start from the matched pair of $\rU(1)$ and the $(a x + b)$-group in $\SU(1, 1)$, we get a double cover of the $\rE(2)$-group as $E$.
The corresponding Poisson structure is essentially the one studied by Maślanka in \cite{MR1267935}.

One interesting feature of this Poisson--Lie group is that the associated cobracket on its Lie algebra is coboundary.
Motivated by this, we look at the matched pairs arising from the Iwasawa decomposition of real simple Lie groups $G$ with finite center whose maximal compact subgroups (which will play the role of $B$) have nondiscrete center.
We show that the cobracket on $E$ is always coboundary (Theorem \ref{thm:gen-e2-from-herm-simpl-cobdry}).

The paper is organized as follows:
in Section \ref{sec:prelim} we collect some preliminary material and fix our conventions.
In Section \ref{sec:PL-grps-from-matched-pairs}, we prove our first main result.
We also include a small discussion on the deformation quantization picture.
In Section \ref{sec:cob-lie-bialg-from-real-simple}, we turn to matched pairs in real simple Lie groups and prove our second main result.

\paragraph{Acknowledgement}
We thank Piotr Stachura for illuminating comments on an early draft of this work, and the anonymous reviewer for their careful reading on the draft and valuable suggestions.

\section{Preliminaries}
\label{sec:prelim}

\subsection{Conventions}

Given a (real) vector space $V$, we denote its linear dual by $V^*$, and the annihilator subspace of a given subspace $W \subset V$ by
\[
W^0 = \{ \phi \in V^* \mid \forall w \in W \colon \phi(w) = 0 \}.
\]
We identify the second exterior power $\medwedge^2 V$ with a subspace of the second tensor power  $\mathcal{T}^2 V$ in such a way that the duality pairing satisfies
\[
\langle x \wedge y, \phi \otimes \psi \rangle = \phi(x) \psi(y) - \psi(x) \phi(y)
\]
for $x, y \in V$ and $\phi, \psi \in V^*$.

For  (real) Lie groups $G$, $A$, etc., we denote their Lie algebras by $\fr{g}$, $\fr{a}$, etc.
The adjoint action of $G$ on $\fr{g}$ is denoted by $\Ad$, and the coadjoint action of $G$ on $\fr{g}^*$ is by $\Ad^\vee$.
The corresponding Lie algebra actions of $\fr g$ are denoted by $\ad$ and $\ad^\vee$, so we have $\ad(x)(y) = [x, y]$ and
\[
\langle \ad^\vee(x)(\phi), y \rangle = - \langle \phi, \ad(x)(y) \rangle = \phi([y,x]).
\]
When $B, C$ are subgroups of $G$, $B C$ denotes the set of elements of the form $b c \in G$ for $b \in B$ and $c \in C$.

\subsection{Matched pairs and Associated Structures}

Let $G$ be a second countable locally compact Hausdorff topological group.
By a \emph{matched pair of subgroups} of $G$, we mean a pair of closed subgroups $B, C$ such that $B \cap C = \{ e \}$ and that $B C$ is open in $G$.
Thus, any element $g$ in the open set $B C \cap C B$ has unique factorizations $g = b c = c' b'$ for $b, b' \in B$ and $c, c' \in C$.

Given such a matched pair, we have a groupoid $\s{G}_{G, B, C}$ (denoted by $\Gamma_B$ in \cite{MR3681683})
defined as follows:
\begin{itemize}
\item base space: $\s{G}_{G, B, C}^{(0)} = B$
\item arrow space: $\s{G}_{G, B, C}^{(1)} = B C \cap C B$
\item range and source maps: $r(g) = b$, $s(g) = b'$ for $g = b c = c' b'$ as above
\item composition: $g \circ g' = b c c''$ (product in $G$) when $g = b c = c' b'$ and $g' = b' c''$
\end{itemize}
The composition is well-defined as we have
\[
b c c'' = g b'^{-1} g' = c' c''' b''
\]
when $g' = b' c'' = c''' b''$.
We will write $\s{G}_B = \s{G}_{G, B, C}$.

If in addition $G$ is a Lie group, $\s{G}_B$ becomes a Lie groupoid.
The Lie algebras form a matched pair \cite{MR1045735}*{Section 4}: $\fr{g} = \fr{b} \oplus \fr{c}$ as vector space, with $\fr b$ and $\fr c$ sitting inside as subalgebras.
Moreover, the groupoid C$^*$-algebras of $\s{G}_B$ make sense as completions of the algebra of compactly supported sections of the half-density bundle \cite{MR636521}, or the one of compactly supported smooth functions and the convolution product with respect to a Haar system \cite{MR1855249}.
If $G$ is a \emph{double Lie group}, i.e., $G = B C$, these algebras are nothing but the crossed products for the corresponding action of $C$ on $C_0(B)$.

In general, when $B C$ is dense in $G$ the partial action of $C$ on $B$, whose graph is $\s{G}_B$, is densely defined so that $C$ acts on $L^\infty(B)$, and the associated crossed product von Neumann algebra $M = C \ltimes L^\infty(B)$ admits the structure of a semi-regular locally compact quantum group, which is regular when $G = B C$ \cite{MR1969723}.
Its dual algebra is given by $\hat{M} = L^\infty(C) \rtimes B$.
In particular, the associated reduced C$^*$-algebras are given by the reduced groupoid C$^*$-algebras $A = C^*_r(\s{G}_B)$ and $\hat{A} = C^*_r(\s{G}_{G, C, B})$.

\subsection{Lie Groupoids and Lie Algebroids}
\label{sec:lie-grpd-lie-algbd}

Let $\s{G}$ be a Lie groupoid with base $M = \s{G}^{(0)}$.
Then we get a Lie algebroid $\s{L}(\s{G})$ on $M$ in the standard way, as follows.
As a vector bundle, it is given by $\s{L} = \ker(\rT r \colon \iota^* \rT \s{G} \to \rT M)$, where $\iota\colon M \to \s G$ is the embedding as identity morphism in $\s{G}$, and $\rT r$ is the tangent map of $r$.
The bracket on $\Gamma(\s{L})$ is given by identifying it with the space of left invariant vector fields on $\s{G}$, and restricting the usual bracket on $\fr{X}(\s{G}) = \Gamma(\rT \s{G})$.
The anchor map $\sm{a}\colon \s{L} \to \rT M$ is the restriction of $\rT s$.

If $G$ is a Lie group and its subgroups $B, C$ form a matched pair in $G$, then the Lie algebroid $\s{L}(\s{G}_B)$ can be modeled on $B \times \fr{c}$.
Namely, at $b \in B$, elements of the fiber $\s{L}(\s{G}_B)_b$ correspond to the tangent vectors at $b$ with integral curve $b \exp(t y)$ for $y \in \fr{c}$.
Equivalently, we use the left translation map $L_b \colon g \mapsto b g$ to identify $\s{L}(\s{G}_B)_b$ with $\fr{c}$.
This gives the trivialization
\begin{equation}
\label{eq:Lie-algebroid-trivialization}
\s{L}(\s{G}_B) \cong B \times \fr{c}.
\end{equation}

\section{Poisson--Lie Groups from Matched Pairs}
\label{sec:PL-grps-from-matched-pairs}

\subsection{Poisson Structures from Lie Algebroids}

Let $G$ be a Lie group, and $B$, $C$ be its subgroups forming a matched pair.
Let us look in detail at the groupoid $\s{G}_B$ over $B$ and the associated Lie algebroid $\s{L}(\s{G}_B)$.

Consider $E = (\rT B)^0$, the (total space of the) subbundle of $\rT^* G|_B = \iota^* \rT^* \s{G}_B$ orthogonal to $\rT B \subset \iota^* \rT \s{G}_B$.
We identify $E$ with $\fr{b}^0 \times B$ by \emph{right} translations.
Using this presentation, we interpret it as the semidirect product for the natural action of $B$ on $\fr{b}^0$.
That is, given $g = (v,b)$ and $h = (w,b^\pr)$ in $\fr{b}^0 \times B$, we put
\begin{align*}
g h &= (v + \Ad^\vee_b w , b b^\pr) , &
g^{-1} &= \left( - \Ad^\vee_{b^{-1}} w , b^{-1} \right).
\end{align*}
This is consistent with viewing $E$ as a subgroup of $\rT^*G$, which has a semidirect product structure $\fr{g}^* \rtimes G$ coming from $\Ad^\vee$.
Let us write the Lie algebra of $E$ as $\fr{e} = \fr{b}^0 \oplus \fr{b}$.

\begin{lem}
The action of $g = (v, b)$ for $\Ad^{E}$ is as follows,
\begin{align}
\label{eq:ADEAZ} \Ad^{E}_g (\psi) &= \Ad^\vee_b \psi && (\psi \in \fr{b}^0),\\
 \label{eq:ADEA} \Ad^{E}_g (y) &= \Ad_b y - \ad^\vee ( \Ad_b y ) (v) && (y \in \fr{b}).
\end{align}
\end{lem}

\begin{proof}
First let us consider $\psi$.
The vector $\Ad^{E}_g (\psi)$ is the differential at $t = 0$ of the integral curve $g(t\psi, e)g^{-1}$.
By the commutativity of $\fr{b}^0$, this is equal to
\[
\left.\dv{t}\right|_{t=0} \left( \Ad^\vee_b t\psi , e \right) = \Ad^\vee_b \psi .
\]

As for $y$, the vector $\Ad^{E}_g (y_i)$ is the differential of the integral curve $g (v, \exp(t y)) g^{-1}$.
Computing the adjoint by $g$, we get
\[
\left( v - \Ad^\vee_{b\exp(t y)b^{-1}} v , b \exp(t y) b^{-1} \right).
\]
Its differential is indeed $\Ad_b y - \ad^\vee( \Ad_b y ) (v)$.
\end{proof}

Note that $E$ is isomorphic to the dual vector bundle of $\s{L}(\s{G}_B)$ by the duality between $\fr{c}$ and $\fr{b}^0$.
Let us describe the induced Poisson structure on $E$ \cite{MR998124}.
The bracket is defined on fiberwise linear functions $\tilde{X}$ on $E$ coming from the sections $X$ of $\s{L}(\s{G}_B)$, and pullbacks $\pi^*(f)$ of the smooth functions on $B$, as
\begin{align}
\label{eq:LGPB}
\left\{ \widetilde{X_1} , \widetilde{X_2} \right\} &= \widetilde{\left[ X_1 , X_2 \right]},&
\left\{ \widetilde{X} , \pi^*(f_1) \right\} &= \pi^*\left( \sm{a}(X) f \right),&
\left\{ \pi^*(f_1) , \pi^*(f_2) \right\} &= 0.
\end{align}
To obtain a more concrete formula for the second relation, let us denote the projections from $\fr{g} \simeq \fr b \oplus \fr c$ to $\fr{b}$ and $\fr{c}$ by $P_{\fr{b}}$ and $P_{\fr{c}}$ respectively.
Then, given $y \in \fr{c}$ and $b \in B$, the corresponding section (up to the trivialization \eqref{eq:Lie-algebroid-trivialization}) $X^L_y$ of $\s L(\s G_B)$ satisfies
\begin{equation}
\label{eq:anchor-map}
\sm{a}(X^L_y)(b) = (\rT R_b)_e P_{\fr{b}} \Ad_b y
\end{equation}
with respect to the right translation map $R_b\colon b^\pr \to b^\pr b$.

Our first goal is to prove the following.

\begin{thm}\label{thm:PLGS}
The Poisson bracket on $E$ characterized by \eqref{eq:LGPB}, together with the semidirect product group structure, defines a Poisson--Lie group structure.
\end{thm}

Let us first start with an observation: the invariance of $\fr{b}$ under $\Ad_b$ for $b \in B$ implies that $\fr{c} \simeq \fr{g} / \fr{b}$ admits an action of $B$, concretely given by the operators $(P_{\fr{c}} \Ad_b)_{b \in B}$.
Another viewpoint is to use the nondegenerate linear duality pairing between $\fr{c}$ and $\fr{b}^0$, given as the restriction of the canonical paring between $\fr{g}$ and $\fr{g}^*$. As $\fr b^0$ is invariant under the transformations $\Ad^\vee_b$ for $b \in B$, we obtain an action of $B$ on $\fr{c}$ by duality, which is the contragredient representation of the above one: for $b \in B$, $y \in \fr{c}$, and $\phi \in \fr{b}^0$, we have
\[
\langle P_{\fr{c}} \Ad_b y, \phi \rangle = \langle y, \Ad_{b^{-1}}^\vee \phi \rangle.
\]

Let us fix a basis $(y_i)_{i \in I}$ of $\fr{c}$, and take its dual basis $(\psi^i)_{i \in I}$ in $\fr{b}^0$ with respect to the duality pairing.
Then the element $t = \sum_{i \in I} \psi^i \otimes y_i$ is invariantly defined, and we have
\begin{equation}
\label{eq:psi-i-y-i-invariance}
\sum_i \Ad^\vee_b \psi^i \ox P_{\fr{c}}\Ad_b y_i = \sum_i \psi^i \ox y_i \quad (b \in B).
\end{equation}

The candidate group $1$-cocycle $E \to \largewedge^2 \fr{e}$ for our bracket is the function $\eta = \eta_0 + \eta_{\fr{b}}$, with the factors
\begin{align*}
\eta_0(g) &= \frac{1}{2} \sum_{i, j} \left\langle v , \Ad_b [ y_i , y_j ] \right\rangle \Ad^\vee_b \psi^i \wedge \Ad^\vee_b \psi^j,&
\eta_{\fr{b}}(g) &= \sum_i \Ad^\vee_b \psi^i \wedge P_{\fr{b}} \Ad_b y_i,
\end{align*}
where we write $g = (v, b)$ as above.

\begin{lem}
\label{lem:Pi-corresp-to-eta}
The Poisson bivector $\Pi \in \Gamma(E, \largewedge^2 \rT E)$ for \eqref{eq:LGPB} is given by $\Pi_g = (\rT R_g)_e^{\ox 2} \eta(g)$.
\end{lem}

\begin{proof}
We have to check
\[
\langle \Pi_g, d f \otimes d f' \rangle = \{f, f'\}(g)
\]
for functions $f, f'$ of the form either $\tilde{X}$ or $\pi^*(f'')$.
The third relation in \eqref{eq:LGPB} is obviously satisfied as $\eta(g)$ does not have components in $\largewedge^2 \fr{b}$.

Note that we have
\begin{equation}
\label{eq:y-tilde-Ad-psi}
\begin{split}
\left\langle (\rT \widetilde{y_i})_g , (\rT R_g)_e \mathopen{} \left( \Ad^\vee_b \psi^j  \right)  \right\rangle
&= \left.\dv{t}\right|_{t=0} \widetilde{y_i} \mathopen{} \left( v + t \Ad^\vee_b \psi^j , b \right)\\
&= \left.\dv{t}\right|_{t=0} \left\langle v + t \Ad^\vee_b \psi^j , \Ad_b y_i \right\rangle\\
&= \left.\dv{t}\right|_{t=0} t \left\langle \psi^j , y_i \right\rangle
= \delta_{i,j}.
\end{split}
\end{equation}
Here, the adjoint action of $b$ on $y_i$ comes from the fact that we are comparing left translates from the Lie algebroid with right translates in $E$.
This, combined with \eqref{eq:anchor-map}, implies the second relation in \eqref{eq:LGPB} for $\Pi$.

Finally, we have
\[
\begin{split}
\left\langle (\rT \widetilde{y_i})_g , (\rT R_g)_e \left( P_{\fr{b}} \Ad_b y_j \right)  \right\rangle &= \left. \dv{t} \right|_{t=0} \left\langle \Ad^\vee_{\exp(t P_{\fr{b}}\Ad_b y_i)} (v) , \Ad_{\exp(t P_{\fr{b}}\Ad_b y_i)a} (y_j) \right\rangle \\
&= \left. \dv{t} \right|_{t=0} \left\langle v , \Ad_b y_j \right\rangle \\
&= 0 .
\end{split}
\]
Combining this with \eqref{eq:y-tilde-Ad-psi}, we obtain the first relation in \eqref{eq:LGPB} for $\Pi$.
\end{proof}

\begin{lem}
With $g \in E$ presented by $(v, b) \in \fr{b}^0 \times B$, we have
\begin{equation}
\eta_0(g) = \frac{1}{2} \left\langle v , \left[ y_i , y_j \right] \right\rangle \psi^i \wedge \psi^j - \Ad^\vee_b \psi^i \wedge \ad^\vee ( P_{\fr{b}}\Ad_b y_i ) (v) . \label{eq:AEFLP}
\end{equation}
\end{lem}

\begin{proof}
The coefficient of $\Ad^\vee_b \psi^i \wedge \Ad^\vee_b \psi^j$ in $\eta_0(g)$ can be written as
\[
\left\langle v , \Ad_b \left[ y_i , y_j \right] \right\rangle = \left\langle v , P_{\fr{c}} \Ad_b \left[ y_i , y_j \right] \right\rangle = \left\langle v , P_{\fr{c}} \left[ \Ad_b y_i , \Ad_b y_j \right] \right\rangle
\]
using the assumption $v\in\fr{b}^0$.
The second term on the right hand side can be expanded as
\begin{multline*}
P_{\fr{c}} \left[ \left( P_{\fr{b}} + P_{\fr{c}} \right) \Ad_b y_i , \left( P_{\fr{b}} + P_{\fr{c}} \right) \Ad_b y_j \right] \\
= \left[ P_{\fr{c}} \Ad_b y_i , P_{\fr{c}} \Ad_b y_j \right] + P_{\fr{c}} \left[ P_{\fr{b}} \Ad_b y_i , P_{\fr{c}} \Ad_b y_j \right] + P_{\fr{c}} \mathopen{} \left[ P_{\fr{c}} \Ad_b y_i , P_{\fr{b}} \Ad_b y_j \right]
\end{multline*}
which we pair with $v$.

Then, using the invariance \eqref{eq:psi-i-y-i-invariance}, we can write
\begin{align*}
\left\langle v , \left[ P_{\fr{c}} \Ad_b y_i , P_{\fr{c}} \Ad_b y_j \right] \right\rangle \Ad^\vee_b \psi^i \wedge \Ad^\vee_b \psi^j &= \left\langle v , \left[ y_i , y_j \right] \right\rangle \psi^i \wedge \psi^j , \\
\left\langle v , \left[ P_{\fr{b}} \Ad_b y_i , P_{\fr{c}} \Ad_b y_j \right] \right\rangle \Ad^\vee_b \psi^i \wedge \Ad^\vee_b \psi^j &= \left\langle v , \left[ P_{\fr{b}} \Ad_b y_i , y_j \right] \right\rangle \Ad^\vee_b \psi^i \wedge \psi^j , \\
\left\langle v , \left[ P_{\fr{c}} \Ad_b y_i , P_{\fr{b}} \Ad_b y_j \right] \right\rangle \Ad^\vee_b \psi^i \wedge \Ad^\vee_b \psi^j &= \left\langle v , \left[ y_i , P_{\fr{b}} \Ad_b y_j \right] \right\rangle \psi^i \wedge \Ad^\vee_b \psi^j .
\end{align*}
Notice that the last two terms are the same after swapping the order of the bracket and the wedge product and renaming dummy indices.

Finally, by definition
\[
\left\langle v , \left[ P_{\fr{b}} \Ad_b y_i , y_j \right] \right\rangle = - \left\langle \ad^\vee (P_{\fr{b}} \Ad_b y_i) (v) , y_j \right\rangle ,
\]
so that we can now complete one of the summations to get that
\[
\left\langle v , \left[ P_{\fr{b}} \Ad_b y_i , y_j \right] \right\rangle \Ad^\vee_b \psi^i \wedge \psi^j = - \Ad^\vee_b \psi^i \wedge \ad^\vee ( P_{\fr{b}} \Ad_b y_i ) (v) .
\]
Combining these we obtain the claim.
\end{proof}

Now we are ready for the proof of our first main result.

\begin{proof}[Proof of Theorem \ref{thm:PLGS}]
By Lemma \ref{lem:Pi-corresp-to-eta}, it is enough to show that $\eta$ satisfies the cocycle condition
\[
\eta(g h) = \eta(g) + \left( \Ad^{E}_g \ox \Ad^{E}_g \right) \eta(h),
\]
or equivalently,
\[
\left( \Ad^{E}_g \ox \Ad^{E}_g \right) \eta(h) = \eta(g h) - \eta(g) .
\]
We proceed by computing the left-hand side for the two pieces $\eta_0$ and $\eta_{\fr{b}}$ separately, and then compare.
In the following we write $g = (v, b)$ and $h = (w, b^\pr)$.

First we expand $\left( \Ad^{E}_g \ox \Ad^{E}_g \right) \eta_0(h)$ as
\begin{multline*}
\frac{1}{2} \left\langle w , \Ad_{b^\pr} \mathopen{}\left[ y_i , y_j \right] \right\rangle \Ad^\vee_{b b^\pr} \psi^i \wedge \Ad^\vee_{b b^\pr} \psi^j 
= \frac{1}{2} \left\langle \Ad^\vee_b w , \Ad_{b b^\pr} \mathopen{} \left[ y_i , y_j \right] \right\rangle \Ad^\vee_{b b^\pr} \psi^i \wedge \Ad^\vee_{b b^\pr} \psi^j \\
= \frac{1}{2} \left\langle v + \Ad^\vee_b w , \Ad_{b b^\pr} \mathopen{} \left[ y_i , y_j \right] \right\rangle \Ad^\vee_{b b^\pr} \psi^i \wedge \Ad^\vee_{b b^\pr} \psi^j - \frac{1}{2} \left\langle v , \Ad_{b b^\pr} \mathopen{} \left[ y_i , y_j \right] \right\rangle \Ad^\vee_{b b^\pr} \psi^i \wedge \Ad^\vee_{b b^\pr} \psi^j .
\end{multline*}
We recognize the first term as $\eta_0(g h)$, so we focus on the second term.
If we apply our formula \eqref{eq:AEFLP} to it, we find
\[
\frac{1}{2} \left\langle v , \Ad_{b b^\pr} \mathopen{} \left[ y_i , y_j \right] \right\rangle \Ad^\vee_{b b^\pr} \psi^i \wedge \Ad^\vee_{b b^\pr} \psi^j = \frac{1}{2} \left\langle v , \left[ y_i , y_j \right] \right\rangle \psi^i \wedge \psi^j - \Ad_{b b^\pr}^* \psi^i \wedge \ad^\vee ( P_{\fr{b}}\Ad_{b b^\pr} y_i ) (v).
\]
The relation \eqref{eq:AEFLP} also implies that
\[
\frac{1}{2} \left\langle v , \left[ y_i , y_j \right] \right\rangle \psi^i \wedge \psi^j = \frac{1}{2} \left\langle v , \Ad_b \mathopen{} \left[ y_i , y_j \right] \right\rangle \Ad^\vee_b \psi^i \wedge \Ad^\vee_b \psi^j + \Ad^\vee_b \psi^i \wedge \ad^\vee ( P_{\fr{b}} \Ad_b y_i ) (v) .
\]
Combining the two yields
\begin{multline*}
\frac{1}{2} \left\langle v , \Ad_{b b^\pr} \mathopen{} \left[ y_i , y_j \right] \right\rangle \Ad^\vee_{b b^\pr} \psi^i \wedge \Ad^\vee_{b b^\pr} \psi^j =\\
 \frac{1}{2} \left\langle v , \Ad_b \mathopen{} \left[ y_i , y_j \right] \right\rangle \Ad^\vee_b \psi^i \wedge \Ad^\vee_b \psi^j + \Ad^\vee_b \psi^i \wedge \ad^\vee ( P_{\fr{b}} \Ad_b y_i ) (v)\\
  - \Ad^\vee_{b b^\pr} \psi^i \wedge \ad^\vee ( P_{\fr{b}} \Ad_{b b^\pr} y_i ) (v) .
\end{multline*}
The first term on the right-hand side is precisely $\eta_0(g)$, so we now have
\begin{multline}\label{eq:g-on-eta-0-h}
\left( \Ad^{E}_g \ox \Ad^{E}_g \right) \eta_0(h) =\\
 \eta_0(g h) - \eta_0(g) - \Ad^\vee_b \psi^i \wedge \ad^\vee ( P_{\fr{b}} \Ad_b y_i ) (v) + \Ad^\vee_{b b^\pr} \psi^i \wedge \ad^\vee ( P_{\fr{b}} \Ad_{b b^\pr} y_i ) (v) .
\end{multline}

Second, we see that $\left( \Ad^{E}_g \ox \Ad^{E}_g \right) \eta_{\fr{b}}(h) $ is equal to
\[
\Ad^\vee_{b b^\pr} \psi^i \wedge \left[ \Ad_b P_{\fr{b}} \Ad_{b^\pr} y_i - \ad^\vee ( \Ad_b P_{\fr{b}} \Ad_{b^\pr} y_i ) (v) \right]
\]
by \eqref{eq:ADEAZ} and \eqref{eq:ADEA}.
To obtain projections in the right places, write
\[
\Ad_b P_{\fr{b}} \Ad_{b^\pr} y_i = P_{\fr{b}} \Ad_b P_{\fr{b}} \Ad_{b^\pr} y_i = P_{\fr{b}} \Ad_b \mathopen{} \left( 1 - P_{\fr{c}} \right) \Ad_{b^\pr} y_i .
\]
We proceed by exploiting the duality again,
\[
\begin{split}
\Ad^\vee_{b b^\pr} \psi^i \wedge \left[ P_{\fr{b}} \Ad_{b b^\pr} y_i - P_{\fr{b}} \Ad_b P_{\fr{c}} \Ad_{b^\pr} y_i \right] &= \Ad^\vee_{b b^\pr} \psi^i \wedge P_{\fr{b}} \Ad_{b b^\pr} y_i - \Ad^\vee_b \psi^i \wedge P_{\fr{b}} \Ad_b y_i \\
&= \eta_{\fr{b}}(g h) - \eta_{\fr{b}}(g) .
\end{split}
\]
This implies
\[
\left( \Ad^{E}_g \ox \Ad^{E}_g \right) \eta_{\fr{b}}(h) = \eta_{\fr{b}}(g h) - \eta_{\fr{b}}(g) - \Ad^\vee_{b b^\pr} \psi^i \wedge \ad^\vee ( P_{\fr{b}} \Ad_{b b^\pr} y_i - P_{\fr{b}} \Ad_b P_{\fr{c}} \Ad_{b^\pr} y_i ) (v),
\]
and by invariance \eqref{eq:psi-i-y-i-invariance} we have
\begin{equation}\label{eq:g-on-eta-b-h}
\begin{split}
\left( \Ad^{E}_g \ox \Ad^{E}_g \right) \eta_{\fr{b}}(h)
&= \eta_{\fr{b}}(g h) - \eta_{\fr{b}}(g) \\
&\quad - \Ad^\vee_{b b^\pr} \psi^i \wedge \ad^\vee ( P_{\fr{b}} \Ad_{b b^\pr} y_i ) (v) + \Ad^\vee_b \psi^i \wedge \ad^\vee ( P_{\fr{b}} \Ad_b y_i ) (v) .
\end{split}
\end{equation}

Combining \eqref{eq:g-on-eta-0-h} and \eqref{eq:g-on-eta-b-h}, we find
\[
\begin{split}
\left( \Ad^{E}_g \ox \Ad^{E}_g \right) \left( \eta_0(h) + \eta_{\fr{b}}(h) \right) &= \eta_0(g h) - \eta_0(g) + \eta_{\fr{b}}(g h) - \eta_{\fr{b}}(g)\\
&\quad  + \Ad^\vee_{b b^\pr} \psi^i \wedge \ad^\vee ( P_{\fr{b}} \Ad_{b b^\pr} y_i ) (v)- \Ad^\vee_{b b^\pr} \psi^i \wedge \ad^\vee ( P_{\fr{b}} \Ad_{b b^\pr} y_i ) (v) \\
&\quad- \Ad^\vee_b \psi^i \wedge \ad^\vee ( P_{\fr{b}} \Ad_b y_i ) (v) + \Ad^\vee_b \psi^i \wedge \ad^\vee ( P_{\fr{b}} \Ad_b y_i ) (v) \\
&= \eta(g h) - \eta(g) ,
\end{split}
\]
as desired.
\end{proof}

\begin{rmk}
Let us sketch another argument to obtain the cocycle property of $\eta$.
We thank the reviewer for this observation.
Generally, suppose that $(\pi, V)$ is a linear representation of a Lie group $G'$, and further suppose we have a linear representation of $G'$ on $V \oplus V^*$ by operators of the form
\[
\begin{bmatrix}
 \pi_g & \xi'(g)\\
 0 & \pi^\vee_g
\end{bmatrix} \quad (g \in G')
\]
for some function $\xi'$ from $G'$ to $\Lin(V^*, V)$, the space of linear maps from $V^*$ to $V$.
Then
\[
\xi'(g h) = \pi_g \xi'(h) + \xi'(g) \pi^\vee_h
\]
implies that the function $\xi''(g) = \xi'(g) (\pi^{\vee}_g)^{-1}$ is a $1$-cocycle for the adjoint representation of $G'$ on $\Lin(V^*, V)$.
Identifying $\Lin(V^*, V)$ with $V \otimes V = \s T^2 V$ and then composing $\xi''$ with the equivariant projection to the $G'$-invariant summand $\medwedge^2 V$, we obtain a $1$-cocycle $\xi$ with values in $\medwedge^2 V$.
Up to rescaling, our construction corresponds to $V = \fr{b}^0 \oplus \fr{b}$ (with $V^*$ identified with $\fr{c} \oplus \fr{c}^0$) and $G' = E$.
To be more precise, $V \oplus V^*$ is identified with $\fr{g}^* \oplus \fr{g}$ up to rearranging summands in
\[
V \oplus V^* = \fr{b}^0 \oplus \fr{b} \oplus \fr{c} \oplus \fr{c}^0.
\]
On $\fr{g}^* \oplus \fr{g}$ one has the adjoint representation of $G'' = \fr{g}^* \rtimes G$, and its restriction on $E = \fr{b}^0 \rtimes B$ has the desired triangular form, since this action of $G''$ restricts to the adjoint representation of $E$ on $V$.
We then obtain $-2 \xi = \eta$ by a direct computation.
\end{rmk}

\subsection{Alternative Proof of Theorem \ref{thm:PLGS} for Double Lie Groups}

It was pointed out to us by P.\ Stachura that Theorem \ref{thm:PLGS} can also be derived from work of Zakrzewski \cite{MR1081011} when the matched pair decomposition is global.
This situation is referred to as a \emph{double Lie group}.
We will freely use the terminology and notation of \cite{MR1764452}.
Now we give a sketch of the argument.

Let $G = BC$ be a double Lie group and consider the groupoids $\s{G}_B$ and $\s{G}_C = \s{G}_{G,C,B}$.
Write $m_C$ for the multiplication relation of $\s{G}_C$.
We begin by taking its transpose, which gives the relation
\begin{equation*}
	m_C^T (g) = \left\{ \left. (b_1 c , c b_2) \right| b_1,b_2 \in B, c \in C, b_1 c b_2 = g \right\} .
\end{equation*}
This can be viewed as a relation $m_C^T \colon \s{G}_B \rightarrowtriangle \s{G}_B \times \s{G}_B$ and turns out to be a (Zakrzewski) morphism of Lie groupoids.

We now apply the phase lift functor to $\s{G}_B$ and the morphism $m_C^T$.
The \emph{phase lift} of a differentiable relation $r \colon X \rightarrowtriangle Y$ is a new relation $\rP r \colon \rT^* X \rightarrowtriangle \rT^* Y$ with graph
\begin{equation*}
	(\alpha,\beta) \in \Gr(\rP r) \Leftrightarrow \forall (u, v) \in T_{(\pi_Y(\alpha),\pi_X(\beta))} \Gr(r) \colon: \langle \alpha , u \rangle = \langle \beta , v \rangle
\end{equation*}
The resulting groupoid $\rP \s{G}_B$ has as arrow space the total space of the cotangent bundle $\rT^* \s{G}_B$.
However, the base space is not $\rT^* B$, as one might expect.
The unit relation $e \colon \{ 1 \} \rightarrowtriangle \s{G}_B$ becomes a relation $\rP e \colon \{ 1 \} \times \{ 0 \} \rightarrowtriangle \rT^* \s{G}_B$.
Because the original base space was $B$, it follows that the conditions for $(g,X^*) \in \rT^* \s{G}_B$ to lie in $\Gr(\rP e)$ are that $g \in B$ and that $X^* \in \fr{b}^0$, as the fiber of the tangent and cotangent bundle of $\{ 1 \}$ is the zero vector space.

We now claim that the base map of the phase lift $\rP m_C^T$ is precisely the group operation of $(\rT B)^0$.
The graph of the base map is the transpose of the intersection
\begin{equation*}
	\Gr( \rP m_C^T ) \cap (\rT B)^0 \times (\rT B)^0 \times (\rT B)^0 \subset \bigl( \rT^* \s{G}_B \times \rT^* \s{G}_B \bigr) \times \rT^* \s{G}_B .
\end{equation*}
Let $\bigl( (b_1, \psi_1) , (b_2, \psi_2) , (b, \psi) \bigr) \in (\rT B)^0 \times (\rT B)^0 \times (\rT B)^0$.
Note that $(b_1,b_2,b)$ lies in the graph of $m_C^T$ if and only if $b = b_1 b_2$.
We describe the tangent space to $\Gr( m_C^T)$ at $(b_1,b_2,b_1 b_2)$.
Any tangent vector in the direction of $B$ is not important, because $\psi_1,\psi_2,\psi$ are from $\fr{b}^0$.
Take $X \in \fr{c}$ instead, then we get a curve
\begin{equation*}
\left( b_1 e^{t X}, e^{t X} b_2, b_1 e^{t X} b_2 \right) = \left( \left( b_1 e^{t X} b_1^{-1} \right) b_1, e^{t X} b_2, \left( b_1 e^{t X} b_1^{-1} \right) b_1 b_2 \right) .
\end{equation*}
Hence we get the tangent vectors $( \Ad_{b_1} (X) , X , \Ad_{b_1} (X) )$, or equivalently $( Y , \Ad_{b_1^{-1}} (Y) , Y )$.
Plugging this into the equation defining the graph of $\rP m_C^T$, we find the condition that
\[
 \psi(Y) = \psi_1(Y) + \psi_2( \Ad_{b_1^{-1}} (Y) ) = \left( \psi_1 + \Ad^\vee_{b_1} \psi_2 \right) (Y) .
\]
Therefore the base map of $\rP m_C^T$ is
\begin{equation*}
\left( (b_1, \psi_1) , (b_2, \psi_2) \right) \mapsto \left( b_1 b_2 , \psi_1 + \Ad^\vee_{b_1} \psi_2 \right) ,
\end{equation*}
as claimed.

Then the phase lift of a Zakrzewski morphism produces a morphism of symplectic groupoids and that the base map of such a morphism is always a Poisson map, see for example \citelist{\cite{MR996653}*{Chapitre II}\cite{MR1081011}*{Section 5}}.

\subsection{Example: the \texorpdfstring{$\rE(2)$}{E(2)} Group from \texorpdfstring{$\SU(1,1)$}{SU(1,1)}}
\label{sec:exmpl-e2-from-su11}

Let us take $G = \SU(1, 1)$, and its Iwasawa decomposition $G = K A N$ with subgroups
\begin{gather*}
\begin{align*}
K &= U(1) = \left\{ \begin{pmatrix} e^{i \vph} & 0 \\ 0 & e^{-i \vph}\end{pmatrix} \Biggm| \vph \in \R \right\},&
A &=  \left\{ \begin{pmatrix} \cosh(t) & \sinh(t) \\ \sinh(t) & \cosh(t) \end{pmatrix} \Biggm| t \in \R \right\},
\end{align*}\\
N = \left\{ \begin{pmatrix} 1 + i s &  - i s \\ i s & 1 - i s \end{pmatrix} \Biggm| s \in \R \right\}.
\end{gather*}
Take $B = K$ and $C = A N$ as matched pair in $G$.
In this case we get $\fr{b}^0 \simeq \C$, and the group $E$ is the semidirect product $\C \rtimes U(1)$ with the product
\[
(z,e^{i\vph})(w,e^{i\psi}) = (z + e^{2i\vph}w,e^{i(\vph+\psi)}).
\]
Denote the semidirect product $\R^2\rtimes \SO(2)$ for the natural rotation action of $\SO(2)$ on $\R^2$ by $\rE^+(2)$, which is the `positive' part of the $2$-dimensional Euclidean group.
Now Consider the two-fold cover of $\rE^+(2)$ given by the matrix group
\[
\rE(2) = \left\{ 
\begin{pmatrix}
v & n \\ 0 & v^{-1}
\end{pmatrix}
\Biggm| v \in \bb{T} , ~ n \in \C \right\} ,
\]
which acts on $\zeta\in\C$ as $\zeta \mapsto v^2\zeta + v n$.
We have an isomorphism $E \simeq \rE(2)$ through the identifications $v = e^{i\vph}$ and $n = e^{-i\vph}z$.
This isomorphism also matches up the actions of the groups on $\C$.
As we will see below, the Poisson--Lie group structure on $\rE(2)$ obtained by our scheme agrees (up to double covering) with the one considered in \cite{MR1267935}.

Write the generators of $K$, $A$, $N$ above as
\begin{align*}
i h &= \begin{pmatrix}
i & 0\\
0 & -i
\end{pmatrix},&
y^{(a)} &= \begin{pmatrix}
0 & 1\\
1 & 0
\end{pmatrix},&
y^{(2)} &= \begin{pmatrix}
i & -i\\
i & -i
\end{pmatrix}.
\end{align*}
We realize $\fr{g}^*$ as a subspace of $\fr{sl}_2(\C)$ by
\[
\fr{g}^* = \left\{ 
\begin{pmatrix}
a & z\\
0 & -a
\end{pmatrix}
\Biggm|
a \in \R, z \in \C
 \right\}
\]
compatible with the natural duality pairing
\[
\langle x, y \rangle = \Im \Tr(x y).
\]
Then $\fr{b}^0$ is spanned by
\begin{align*}
y^{(a)\ast} &= 
\begin{pmatrix}
0 & i\\
0 & 0
\end{pmatrix},&
y^{(2)\ast} = 
\begin{pmatrix}
0 & 1\\
0 & 0
\end{pmatrix}.
\end{align*}
To help comparison with \cite{MR1267935}, let us write our basis of $\fr{e}$ as
\begin{align*}
P_1 &= y^{(a)\ast},&
P_2 &= y^{(2)\ast},&
J = i h.
\end{align*}
Their relations are given by
\begin{align*}
\left[ P_1 , P_2 \right] &= 0,&
\left[ J , P_1 \right] &= 2 P_2,&
\left[ J , P_2 \right] &= - 2 P_1.
\end{align*}

Let us compute the bracket on $\fr{e}^*$ induced by the Poisson bracket coming from the Lie groupoid.
As before we denote by $X^L_y$ the sections of $\s{L}(\s{G}_B)$ corresponding to $y \in \fr{c}$.
We will also denote the sections coming from $y^{(a)}$ and $y^{(2)}$ by $X^L_{(a)}$ and $X^L_{(2)}$.
Furthermore, we denote the corresponding fiber-wise linear functions on $E$ by $\widetilde{X^L_{y}}$, etc.
Concretely, we have
\begin{align*}
\widetilde{X^L_{(a)}} (P_1,e^{i\vph}) &= \langle P_1, \Ad_{e^{i\vph}} y^{(a)} \rangle = \cos(2\vph), &
\widetilde{X^L_{(a)}} (P_2,e^{i\vph}) &= \sin(2\vph),\\
\widetilde{X^L_{(2)}} (P_1,e^{i\vph}) &= -\sin(2\vph), &
\widetilde{X^L_{(2)}} (P_2,e^{i\vph}) &= \cos(2\vph) .
\end{align*}
So these two functions can and will be used as coordinate functions $p_1$ and $p_2$ on the `linear part' of $E$.
On the $U(1)$ part of $E$, we have the function
\[
\mathrm{U}(1) \to \C, \quad  \begin{pmatrix}
 e^{i\vph} & 0 \\ 0 & e^{-i\vph} 
\end{pmatrix}
 \mapsto e^{i\vph},
\]
which we denote by $e^{i \vph}$ again.

The anchor map $\sm{a} \colon \s{L}( \s{G}_K ) \rightarrow \rT K = \fr{k} \times K$ becomes
\[
\sm{a}( X^L_y ) (e^{i\vph}) = (P_{\fr{k}} \mathopen{} \left( \mathrm{Ad}_{e^{i\vph}} y \right), e^{i\vph}),
\]
where $P_{\fr{k}}$ is the projection $\fr{g} \to \fr{k}$ corresponding to the decomposition $\fr{g} = \fr{k} \oplus (\fr{a} \oplus \fr{n})$.
For our basis above, we get
\begin{align*}
\sm{a}( X^L_{(a)} ) (e^{i\vph}) &= (\sin(2\vph) J, e^{i \vph}),&
\sm{a}( X^L_{(2)} ) (e^{i\vph}) &= (( 1 - \cos(2\vph)) J, e^{i \vph}).
\end{align*}

We can now compute the Poisson structure.
The bracket on the linear part is given by
\[
\left\{\widetilde{X^L_{(a)}} , \widetilde{X^L_{(2)}} \right\} = \widetilde{X^L_{[y^{(a)}, y^{(2)}]}} = 2 \widetilde{X^L_{(2)}}.
\]
Between the linear part and the base part, we have
\begin{align*}
\left\{ \widetilde{X^L_{(a)}} , e^{i\vph} \right\} &= \sm{a} (X^L_{(a)}) e^{i\vph} = \sin(2\vph) J^R e^{i\vph} = i \sin(2\vph) e^{i\vph},&
\left\{ \widetilde{X^L_{(2)}} , e^{i\vph} \right\} &= i \mathopen{} \left( 1 - \cos(2\vph) \right) e^{i\vph} .
\end{align*}

To compare with \cite{MR1267935}, we pass to the quotient $E/\bb{Z}_2 \simeq \rE^+(2)$.
This identification is given by $(z,e^{i\vph}) \mapsto (z,e^{i2\vph})$, so that the linear functions are unchanged, but the function $e^{i\vth}$ on $\rE^+(2)$ corresponds to $e^{i2\vph}$ on $E$.
As
\begin{align*}
\left\{ \widetilde{X^L_{(a)}} , e^{i2\vph} \right\} &= 2 e^{i\vph} \mathopen{} \left\{ \widetilde{X^L_{(a)}} , e^{i\vph} \right\} = 2 i \sin(2\vph) e^{i2\vph} ,&
\left\{ \widetilde{X^L_{(2)}} , e^{i2\vph} \right\} &= 2 i \left( 1 -\cos(2\vph) \right) e^{i2\vph} ,
\end{align*}
we obtain the following Poisson bracket on $\rE^+(2)$:
\begin{align*}
\{ \widetilde{X^L_{(a)}}, \widetilde{X^L_{(2)}} \} &= 2 \widetilde{X^L_{(2)}} , &
\left\{ \widetilde{X^L_{(a)}} , e^{i\vth} \right\} &= 2 i \sin(\vth) e^{i\vth} ,&
\left\{ \widetilde{X^L_{(2)}} , e^{i\vth} \right\} &= 2 i \left( 1 - \cos(\vth) \right) e^{i\vth} .
\end{align*}
Then, putting
\begin{align*}
V^1 &= -\widetilde{X^L_{(a)}},&
V^2 &=\widetilde{X^L_{(2)}},
\end{align*}
we obtain the bracket \cite{MR1267935}*{Equation (10)} for $\omega = -2$.

In \cite{MR1096123}*{Section 3}, Woronowicz constructed the quantum group $E_q(2)$ starting with the model E$(2)$ for the group of Euclidean motions of the plane.
Note that the Poisson--Lie structure obtained here is different from the one behind Woronowicz's example.
In the notation of \cite{MR1399267}, what we obtained is a scaled version of the cobracket $\delta_3$ on $\fr{e}(2)$, while Woronowicz's $E_q(2)$ corresponds to the family $\delta_1$.
Nevertheless, there is still some similarity between the two.
If we pass to the dual $\fr{e}(2)^*$, with basis $\{J^*,P_1^*,P_2^*\}$, then the Lie brackets dualizing $\delta_1$ and $\delta_3$ yield isomorphic Lie algebras.
Indeed, $\delta_1$ gives
\begin{align}
\label{eq:dual-bracket-for-delta-1}
\left[ P_1^* , P_2^* \right]_1 &= 0 ,&
\left[ P_1^* , J^* \right]_1 &= s P_1^* ,&
\left[ P_2^* , J^* \right]_1 &= s P_2^*,
\end{align}
where $s$ is an auxiliary parameter, while $\delta_3$ gives
\begin{align*}
\left[ P_1^* , P_2^* \right]_3 &= P_2^* ,&
\left[ P_1^* , J^* \right]_3 &= J^* ,&
 \left[ P_2^* , J^* \right]_3 &= 0 .
\end{align*}
So the linear isomorphism $\rho\colon \fr{e}(2)^* \rightarrow \fr{e}(2)^*$ given by
\begin{align*}
\rho(J^*) &= - s P_1^*,&
\rho(P_1^*) &= J^*,&
\rho(P_2^*) &= P_2^*
\end{align*}
satisfies $\rho \circ [\cdot,\cdot]_1 = [\rho(\cdot),\rho(\cdot)]_3$.

Let us also note that this Lie algebra structure on $\fr{e}(2)^*$ is isomorphic to the standard Lie algebra structure on $\fr{su}(1,1)^* = \fr{su}(2)^*$.
The latter is spanned by $h$, $e$, and $i e$ as a real Lie subalgebra of $\fr{sl}_2(\C)$, which is indeed isomorphic to $\fr e(2)^*$ with bracket \eqref{eq:dual-bracket-for-delta-1}.
In particular, dualizing the Lie bialgebras $(\fr{e}(2), \delta_1)$, $(\fr{e}(2), \delta_3)$, and $\fr{su}(1,1)$, we get different Lie bialgebra structures on $\fr{e}(2)^*$.
The ones from $\delta_1$ and $\fr{su}(1,1)$ are cohomologous as we will see in Section \ref{sec:cob-lie-bialg-from-real-simple}, while the one from $\delta_3$ has a different class in cohomology.

\subsection{Deformation Quantization}

Let us explain an analogue of strict deformation quantization, in the framework of unbounded multipliers.
The bounded picture, that is, a C$^*$-algebraic strict deformation quantization in the sense of Rieffel \cite{MR1292010}, is already provided in \cite{MR1722129}.

Let $U(\fr{c})$ be the complexified universal enveloping algebra of $\fr{c}$, i.e., the universal associative unital $\C$-algebra generated by a copy of $\fr{c}$ as its real subspace, with relations $x y - y x = [x, y]_{\fr{c}}$ for $x, y \in \fr{c}$.
The Hopf algebra $U(\fr{c})$ acts on $C^\infty_b(B)$, and the elements of the crossed product $U(\fr{c}) \ltimes C^\infty_b(B)$ can be regarded as unbounded multipliers of $C^*(\s G_B)$, see Appendix \ref{appdix:grpd-C-star-alg}.

Since $\Sym(\fr{c})$ can be identified with the space of complex polynomial functions on $\fr{b}^0$, the elements of $\Sym(\fr{c}) \otimes C^\infty_b(B)$ can be regarded as unbounded multipliers of $C_0(E)$.
Let us consider the corresponding bracket on $\Sym(\fr{c}) \otimes C^\infty_b(B)$.

By choosing an ordered basis of $\fr{c}$, by the same argument as the Poincaré--Birkhoff--Witt Theorem, we get a linear isomorphism
\[
Q\colon \Sym(\fr{c}) \otimes C^\infty_b(B) \to U(\fr{c}) \ltimes C^\infty_b(B)
\]
which is compatible with the filtrations by degree in $\fr{c}$.
For an auxiliary parameter $h$, define $Q_h\colon \Sym(\fr{c}) \otimes C^\infty_b(B) \to U(\fr{c}) \ltimes C^\infty_b(B)$ by
\[
Q_h(y_1 \dots y_k \otimes f) = h^k Q(y_1 \dots y_k \otimes f).
\]

\begin{prop}
We have
\[
\frac{[Q_h(T \otimes f), Q_h(T' \otimes f')]}{h} = Q_h(\{T \otimes f, T' \otimes f'\}) + O(h)
\]
for $T, T' \in \Sym(\fr{c})$ and $f, f' \in C^\infty_b(B)$.
\end{prop}

\begin{proof}
For linear polynomials, we have
\[
\frac{[Q_h(y \otimes 1), Q_h(y' \otimes 1)]}{h} = Q_h(\{y \otimes 1, y' \otimes 1\})\quad (y, y' \in \fr{c})
\]
with exact equality, from the structure of $U(\fr{c})$.
($y \otimes 1$ corresponds to the function $\widetilde{X^L_y}$.)
For general elements $T, T' \in \Sym(\fr{c})$, induction on degree gives
\[
\frac{[Q_h(T \otimes 1), Q_h(T' \otimes 1)]}{h} = Q_h(\{T \otimes 1, T' \otimes 1\}) + O(h^2).
\]

Between $\fr{c}$ and $C^\infty_b(B)$, we again have
\[
\frac{[Q_h(y \otimes 1), Q_h(1 \otimes f)]}{h} = Q_h(\{y \otimes 1, 1 \otimes f\})\quad(y \in \fr{c}, f \in C^\infty_b(B))
\]
with exact equality, from the way the anchor map is defined.
($1 \otimes f$ corresponds to the function $\pi^* f$.)
For general $T \in \Sym(\fr{c})$, again by induction on degree we get
\[
\frac{[Q_h(T \otimes 1), Q_h(1 \otimes f)]}{h} = Q_h(\{T \otimes 1, 1 \otimes f\}) + O(h).
\]

Finally, we also have $[Q_h(1 \otimes f), Q_h(1 \otimes f')] = 0$.
\end{proof}

There is a structure of a Hopf algebra (up to completion) on $U(\fr{c}) \ltimes C^\infty_b(B)$ corresponding to the bicrossed product structure, as follows.
On $C^\infty_b(B)$ we consider the one coming from the group structure of $B$, implemented as
\[
\Delta\colon C^\infty_b(B) \to C^\infty_b(B \times B) \subset \s{M}(C_0(B) \otimes C_0(B)).
\]
(We can also take $C^\infty_b(B)$ as the domain.)
We mix this with the usual cocommutative coproduct $\Delta \colon U(\fr{c}) \to U(\fr{c}) \otimes U(\fr{c})$, using the action of $B$ on $\fr{c}$, as follows.

Note that $U(\fr{c}) \otimes C^\infty(B \times B)$ is identified with the space of smooth functions $B \times B \to U(\fr{c})$ with finite dimensional images.
For $T \otimes f$ with $T \in U(\fr{c})$ and $f \in C^\infty(B \times B)$, denote by $\theta(T \otimes f)$ the function $(g, h) \mapsto T^{g} f(g, h)$, where $T^g$ is the \emph{right} action of $B$ on $U(\fr{c})$ induced by the right action on $\fr{c}$.
Since $U(\fr{c})$ is the increasing union of finite dimensional representations of $B$, $\theta$ is well-defined as a transform on $U(\fr{c}) \otimes C^\infty(B \times B)$.
We then put
\[
\Delta(T \otimes f) = (T_{(1)} \otimes \theta(T_{(2)} \otimes \Delta(f)))_{1324} \quad (T \in U(\fr{c}), f \in C^\infty_b(B)),
\]
where $T_{(1)} \otimes T_{(2)}$ is the Sweedler notation for $\Delta(T)$.

If $B$ is compact, we can have a model of a genuine Hopf algebra by taking the algebra $\s{O}(B)$ of matrix coefficients of finite dimensional complex linear representations instead of $C^\infty(B)$.
Indeed, since $U(\fr{c})$ is a union of finite dimensional $B$-modules, the above $\Delta$ restricts to a coproduct map from $\s H = U(\fr c) \otimes \s O(B)$ to the algebraic tensor product $\s H \otimes \s H$.

\section{Coboundary Lie Bialgebras from Real Simple Lie Groups}
\label{sec:cob-lie-bialg-from-real-simple}

Throughout this section, let $G$ denote a connected real simple Lie group with finite center, and $K$ its maximal compact subgroup.
We further assume that $K$ has non-discrete center, which is equivalent to $\s{Z}(\fr{k})$ being $1$-dimensional.
This means that $G$ is, up to a finite cover, the identity component of the group of isometries on a noncompact irreducible Hermitian symmetric space \cite{MR1834454}.
For example, we could take $G = \SU(p,q)$ and $K = \mathrm{S}(\mathrm{U}(p)\times\mathrm{U}(q))$.

\subsection{Finding the \texorpdfstring{$r$}{r}-Matrix}

Denote the real Lie algebras of $G$ and $K$ by $\fr{g}$ and $\fr{k}$ respectively.
Let us take the Cartan decomposition for $K \subset G$, i.e., an orthogonal decomposition $\fr{g} = \fr{k} \oplus \fr{p}$ for the invariant symmetric bilinear form.
We note
\begin{align}
\label{eq:cobr-rel-for-Cartan-decomp}
[\fr{k},\fr{k}] &\subset \fr{k},&
[\fr{k},\fr{p}] &\subset \fr{p},&
[\fr{p},\fr{p}] &\subset \fr{k} .
\end{align}

Under our assumptions, $\s{Z}(\fr{k})$ must be $1$-dimensional, and we can pick a spanning element $z$ of $\s{Z}(\fr{k})$ satisfying
\begin{equation}
\label{eq:z-normalization}
\left. \ad(z)^2 \right|_{\fr{p}} = - 1,
\end{equation}
see for example \cite{MR1834454}*{Chapter VIII}.

The Iwasawa decomposition $G = K A N$ gives a matched pair of subgroups $K$ and $S = A N$ in $G$, and induces the decomposition $\fr{g} = \fr{k} \oplus \fr{s}$ as a vector space.
We denote the two projections associated to this decomposition by $P^I_{\fr{k}}$ and $P^I_{\fr{s}}$ respectively.
By the recipe of our main result, we get a Poisson--Lie group structure on $E = \fr{k}^0 \rtimes K$, or equivalently a Lie bialgebra structure on $\fr{e} = \fr{k}^0 \rtimes \fr{k}$.
As in \ref{lem:Pi-corresp-to-eta}, the Poisson--Lie structure can be described by the $1$-cocycle $\eta\colon E \rightarrow \fr{e} \ox \fr{e}$.
The cobracket $\delta$ on $\fr{e}$ is then given by differentiating $\eta$.
Let $\{y_i\}$ be a basis of $\fr{s}$ and $\{\psi^i\}$ be the dual basis inside $\fr{k}^0$.
Then, for $x\in \fr{k}$, we get
\begin{align*}
    \delta(x) &= \left.\dv{t}\right|_{t=0} \eta(0,e^{tx}) \\
    &= \left.\dv{t}\right|_{t=0} \sum_i \Ad^\vee_{e^{tx}} \psi^i \wedge P^I_{\fr{k}}\Ad_{e^{tx}} y_i \\
    &= \sum_i P^I_{\fr{k}} \left[ y_i , x \right] \wedge \psi^i ,
\end{align*}
where we have used the Leibniz rule and $P^I_{\fr{k}} y_i = 0$ at the last step.

Similarly, for $\psi \in \fr{k}^0$, $\delta(\psi)$ is an element of $\fr{k}^0 \ox \fr{k}^0$  characterized by 
\begin{equation*}
 \left\langle \delta(\psi) , y \ox y^\pr \right\rangle = \left\langle \psi , \left[ y , y^\pr \right] \right\rangle \quad(y, y^\pr \in \fr{s}).
\end{equation*}

In the following, we will use the notation
\begin{equation*}
    x.(y\ox z) = \left( \ad(x) \ox \iota + \iota \ox \ad(x) \right) (y\ox z) 
\end{equation*}
for $x, y, z \in \fr{e}$.

\begin{thm}
\label{thm:gen-e2-from-herm-simpl-cobdry}
Under the above setting, the Lie bialgebra structure on the Lie algebra $\fr{e}$ is coboundary.
More precisely, $r = z.\delta(z)$ satisfies $\delta(x) = [r, \Delta(x)]$ for all $x \in \fr{e}$.
\end{thm}

\begin{proof}
Let us first check
\[
\delta(x) = \left[ r , \Delta(x) \right]
\]
for $x \in \fr{k}$.

On the one hand, we have
\[
z. z. \delta(x) = - \delta(x).
\]
Indeed, we can identify $\fr{k}^0$ with $\fr{p}$ as representations of $\fr{k}$ by the Killing form of $\fr{g}$.
Then the centrality of $z$ and the normalization condition \eqref{eq:z-normalization} implies this claim.

On the other, we have
\[
0 = \delta\mathopen{}\left( \left[ x,z \right] \right) = x.\delta(z) - z.\delta(x) \quad (x \in \fr k)
\]
by the centrality of $z$ and the cocycle condition for $\delta$.

Combining these two, we obtain
\[
\delta(x) = - z. z. \delta(x) = - x. z. \delta(z) = \left[ z.\delta(z) , \Delta(x) \right] .
\]

It remains to show that
\[
\delta(\psi) = \left[ r , \Delta(\psi) \right]
\]
for $\psi\in\fr{k}^0$.

Let us look at the Cartan decomposition $\fr g = \fr k \oplus \fr p$, and write the corresponding projections as $P^C_{\fr k}$ and $P^C_{\fr{p}}$.
Then For $y \in \fr s$, we have
\[
P^C_{\fr k} y = - P^I_{\fr k} P^C_{\fr p} y ,
\]
by the following general lemma.
\renewcommand{\qedsymbol}{}
\end{proof}

\begin{lem}
\label{lem:PCk-is-PIk-PCp}
Let $V$ be a vector space and suppose that we have decompositions $U \oplus W_1$ and $U \oplus W_2$.
Denote the corresponding projections by $P^1_U, P^1_{W_1}$ and $P^2_U, P^2_{W_2}$ respectively.
Then we have
\begin{equation*}
 P^2_U w = - P^1_U P^2_{W_2} w \quad (w \in W_1).
\end{equation*}
\end{lem}

\begin{proof}
We can expand $0 = P^1_U w$ as $P^1_U P^2_U w + P^1_U P^2_{W_2}$, which is equal to $P^2_U w + P^1_U P^2_{W_2} w$.
\end{proof}

\begin{proof}[Proof of Theorem \ref{thm:gen-e2-from-herm-simpl-cobdry}, continued]
We claim that our candidate for the $r$-matrix can be also written as
\begin{equation}
\label{eq:rmat}
r = \sum_i P^C_{\fr k} y_i \wedge \psi^i .
\end{equation}
By the centrality of $z$, we can write $r = z.\delta(z)$ as
\[
 z.\biggl(\sum_i P^I_{\fr k} [y_i, z] \wedge \psi^i\biggr) = \sum_i P^I_{\fr k} [y_i, z] \wedge \ad^\vee(z)(\psi^i).
\]
Then the $K$-invariance of the canonical tensor $\sum_i y_i \ox \psi^i$ as in \eqref{eq:psi-i-y-i-invariance} implies that the right hand side is equal to
\[
 -\sum_i P^I_{\fr k}[P^I_{\fr s} \ad(z)(y_i), z] \wedge \psi_i = \sum_i P^I_{\fr k} \ad(z)^2(y_i) \wedge \psi_i.
\]
Using \eqref{eq:z-normalization}, we have $\ad(z)^2(y) = -P^C_{\fr p} y$ for $y \in \fr s$.
Then Lemma \ref{lem:PCk-is-PIk-PCp} implies the claim.

Finally, notice that \eqref{eq:cobr-rel-for-Cartan-decomp} implies
\[
\psi\mathopen{}\left( \left[ y , y^\pr \right] \right) = \psi\mathopen{}\left( \left[ P^C_{\fr k} y , y^\pr \right] - \left[ P^C_{\fr k} y^\pr , y \right] \right)
\]
for $\psi \in \fr k^0$.
Therefore,
\[
\begin{split}
\left\langle \psi , \left[ y , y^\pr \right] \right\rangle &= \left\langle \psi , \left[ P^C_{\fr k} y , y^\pr \right] - \left[ P^C_{\fr k} y^\pr , y \right] \right\rangle \\
&= \sum_i \left\langle \psi , \left[ \psi^i(y) P^C_{\fr k} y_i , y^\pr \right] - \left[ \psi^i(y^\pr) P^C_{\fr k} y_i , y \right] \right\rangle \\
&= - \sum_i \left\langle \psi^i \ox \ad^\vee\mathopen{}\left( P^C_{\fr k} y_i \right)(\psi) , y \wedge y^\pr \right\rangle,
\end{split}
\]
but this is nothing but $\left\langle \left[ r , \Delta \psi \right] , y \ox y^\pr \right\rangle$ by \eqref{eq:rmat}.
\end{proof}

\begin{rmk}
We note that the above element $r$ is the only one satisfying $\delta(x) = [r, x]$ for al $x \in \fr{e}$.
Indeed, take any other potential $r$-matrix $r^\pr$, then the difference $r-r^\pr$ would be an invariant element of $\fr{k} \ox \fr{k}^0 \oplus \fr{k}^0 \ox \fr{k}$.
Testing against $\Delta(z)$ shows that no such elements exist except for $0$.
\end{rmk}

\begin{rmk}
While our argument relies on the centrality of $z$, the expression \eqref{eq:rmat} makes sense as an element of $\medwedge^2 \fr{e}$ in general, and in some cases still implements the cobracket.
For example, this is the case for $\fr g = \fr{so}(3,1)$ \cite{MR1280372}, see also a related phenomenon for the semidirect product of $\fr{so}(p,q)$ and $\R^{p + q}$ \citelist{\cite{MR1463043}\cite{arXiv:q-alg/9712040}}.
\end{rmk}

\begin{ex}
Starting from $\fr{g} = \fr{su}(1,1)$ and $\fr{k} = \fr{so}(2)$, we get $r = J \wedge P_2$ for the cobracket on $\fr{e}(2)$, cf.~\cite{MR1267935}.
\end{ex}

\subsection{Example: \texorpdfstring{$G = \SU(p,1)$}{G = SU(p, 1)}}

The Lie algebra of $\SU(p,1)$ for $p\geq 2$ is
\[
\fr{su}(p,1) = \left\{
\begin{pmatrix}
M & b \\
b^* & - \Tr[M]
\end{pmatrix}
\Biggm| M\in M_p(\C) , ~ M^* = - M , ~ b \in \C^p \right\} .
\]
We pick the maximal compact subalgebra
\[
\fr{k} = \left\{ 
\begin{pmatrix}
M & 0 \\
0 & - \Tr[M]
\end{pmatrix}
\Biggm| M\in M_p(\C) , ~ M^* = - M  \right\} ,
\]
which corresponds to the maximal compact subgroup S$(\mathrm{U}(p)\times\mathrm{U}(1))$ of $\SU(p,1)$.
We remark that the center of $\fr{k}$ is one-dimensional and spanned by
\[
z = \frac{1}{p+1} \begin{pmatrix}
i I_p & 0\\
0 & - p i
\end{pmatrix},
\]
where we have normalized such that $\ad^\vee(z)^2|_{\fr{k}^0} = - \iota$.
The associated Cartan decomposition is give by
\[
\fr{su}(p,1) = \fr{k} \oplus \fr{p} , ~ \fr{p} = \left\{
\begin{pmatrix}
0 & b \\
b^* & 0 \end{pmatrix}
\Biggm| b \in \C^p \right\} .
\]

The other decomposition we will use is the Iwasawa decomposition.
This begins with a choice of maximal commutative subalgebra $\fr{a}$ of $\fr{p}$.
We will work with
\[
\fr{a} = \left\{ t
\begin{pmatrix}
0 & e_p \\
e_p^* & 0
\end{pmatrix}
\Biggm| t\in\R \right\},
\]
where $e_p$ is the column vector of size $p$ with entries $0, \dots, 0, 1$.
In addition we have the positive restricted root spaces
\begin{align*}
\fr{g}_{f_1} &= \left\{
\begin{pmatrix}
0 & v & - v \\
- v^* & 0 & 0 \\
- v^* & 0 & 0 
\end{pmatrix}
\Biggm| v \in \C^{p-1} \right\} ,&
\fr{g}_{2f_1} &= \left\{ t
\begin{pmatrix}
0 & 0 & 0 \\
0 & i & -i \\
0 & i & -i \end{pmatrix}
\Biggm| t \in \R \right\} .
\end{align*}
These combine into the nilpotent part $\fr{n} = \fr{g}_{f_1} \oplus \fr{g}_{2f_1}$, which is normalized by $\fr{a}$.
Finally, this gives the solvable part of the Iwasawa decomposition,
\[
\fr{s} = \fr{a} \oplus \fr{n} .
\]

The complexification of $\fr{su}(p,1)$ is $\fr{sl}(p+1,\C)$, and its dual $\fr{su}(p,1)^*$ (with respect to the imaginary part of the Killing form) is given by the Lie algebra of upper triangular $(p+1)\times (p+1)$-matrices with real entries on the diagonal.
Then $\fr{k}^0  \subset \fr{su}(p,1)^*$ is given by
\[
\fr{k}^0 = \left\{ 
\begin{pmatrix}
0 & w \\ 0 & 0
\end{pmatrix}
\Biggm| w \in \C^p \right\},
\]
which is indeed a commutative subspace.

Let us take a basis of $\fr{s}$ as follows:
\begin{align*}
y^{(a)} &= \begin{pmatrix}
 0 & e_p \\ e_p^* & 0 
\end{pmatrix}
 \in \fr{a} ,&
y^{(2)} &= \begin{pmatrix}
 0 & 0 & 0 \\ 0 & i & -i \\ 0 & i & -i 
\end{pmatrix}
 \in \fr{g}_{2f_1} , \\
y^{(R)}_k &= \begin{pmatrix}
 0 & - e_k & e_k \\ e_k^* & 0 & 0 \\ e_k^* & 0 & 0 
\end{pmatrix}
 \in \fr{g}_{f_1},&
y^{(I)}_k &= \begin{pmatrix}
 0 & i e_k & - i e_k \\ i e_k^* & 0 & 0 \\ i e_k^* & 0 & 0 
\end{pmatrix}
 \in \fr{g}_{f_1} \quad (k = 1 , \dots , p -1).
\end{align*}
Their brackets are given by
\begin{align*}
[y^{(a)}, y^{(2)}] &= 2 y^{(2)}, &
[y^{(a)}, y^{(R)}_k] &= y^{(R)}_k,&
[y^{(a)}, y^{(I)}_k] &= y^{(I)}_k,\\
\left[ y^{(R)}_k , y^{(R)}_\ell \right] &= 0 , &
\left[ y^{(I)}_k , y^{(I)}_\ell \right] &= 0 , &
\left[ y^{(R)}_k , y^{(I)}_\ell \right] &= 2 \delta_{k\ell} y^{(2)},
\end{align*}
while all other brackets involving $y^{(2)}$ vanish.

The dual basis in $\fr k^0$ is given by
\begin{align*}
y^{(a)\ast} &= \begin{pmatrix}
 0 & i e_p \\ 0 & 0 
\end{pmatrix}
 , &
y^{(2)\ast} &= \begin{pmatrix}
 0 & e_p \\ 0 & 0 
\end{pmatrix}
 , \\
y^{(R)\ast}_k &= \begin{pmatrix}
 0 & i e_k \\ 0 & 0 
\end{pmatrix}
 , &
y^{(I)\ast}_k &= \begin{pmatrix}
 0 & e_k \\ 0 & 0 
\end{pmatrix}
 \quad (k = 1 , \dots , p -1).
\end{align*}
From above computation, we see that the induced cobracket $\delta$ on $\fr{k}^0$ is given by
\begin{align*}
\delta\mathopen{}\left( y^{(a)\ast} \right) &= 0 , &
\delta\mathopen{}\left( y^{(\epsilon)\ast}_k \right) &= y^{(a)\ast} \wedge y^{(\epsilon)\ast}_k , &
\delta\mathopen{}\left( y^{(2)\ast} \right) &= y^{(a)\ast} \wedge y^{(2)\ast} + 2 \sum_{k=1}^{p-1} y^{(R)\ast}_k \wedge y^{(I)\ast}_k,
\end{align*}
where $\epsilon$ denotes $R$ or $I$ in the second relation.
From \eqref{eq:rmat}, we see that the $r$-matrix is given by
\[
\begin{split}
r &= P^C_{\fr k} y^{(2)} \wedge y^{(2)\ast} + \sum_{k=1}^{p-1} \left[ P^C_{\fr k} y^{(R)}_k \wedge y^{(R)\ast}_k + P^C_{\fr k} y^{(I)}_k \wedge y^{(I)\ast}_k \right] \\
&= \left( i e_{pp} - i e_{p+1,p+1} \right) \wedge y^{(2)\ast} + \sum_{k=1}^{p-1} \left[ ( e_{pk} - e_{kp} ) \wedge y^{(R)\ast}_k + ( i e_{kp} + i e_{pk} ) \wedge y^{(I)\ast}_k \right] .
\end{split}
\]

Let us take a look at the structure one the Lie group level.
(The formulas apply to the case $p = 1$ as well.)
The maximal compact subgroup $K$ is
\[
\mathrm{S} ( \mathrm{U}(p) \times \mathrm{U}(1) ) = \left\{ 
\begin{pmatrix}
U & 0 \\ 0 & \ol{\det(U)}
\end{pmatrix}
\Biggm| U \in \mathrm{U}(p) \right\} \cong \mathrm{U}(p) .
\]
Write $\rE(p+1)$ for the Poisson--Lie group $\fr k^0 \rtimes K$ constructed out of $\SU(p,1)$ as in the previous section.
Then $\rE^\C(p+1) \cong \C^p \rtimes \mathrm{U}(p)$ and a brief computation shows that
\[
\Ad^\vee_U (z_1,\dots,z_p) = \det(U) U \begin{pmatrix}
z_1 \\ \vdots \\ z_p 
\end{pmatrix}
.
\]
Thus, the groups $\rE^\C(p+1)$ can be regarded as subgroups of the even-dimensional Euclidean groups compatible with the complex structure.

\subsection{Another Deformation Scheme}

Let $\fr{g}_\C$ be the complexification of $\fr{g}$.
By our assumption on $K$, there is commutative subgroup $T \subset K$ whose complexified Lie algebra $\fr{t}_\C$ is a Cartan subalgebra of $\fr{g}_\C$.
Let us also choose an order of roots for this Cartan subalgebra.
Then we get a Manin triple $(\fr{g}_\C,\fr{g},\fr{g}^*)$, where $\fr{g}^*$ the subspace of $\fr{g}_\C$ spanned by $i\fr{t}$ and the positive root spaces, and the real invariant bilinear form on $\fr{g}_\C$ is given by the imaginary part of the Killing form \cite{MR1284791}.
This defines a Poisson--Lie group structure on $G$.

As remarked at the end of Section \ref{sec:exmpl-e2-from-su11}, in the case of $\fr g = \fr{su}(1,1)$, the Lie algebra $\fr{e}^*$ is isomorphic to $\fr{g}^*$.
In general this relation does not hold, which can be directly seen for the above examples $\fr{su}(p, 1)$.
Nonetheless, there is another similar construction that does relate $\fr g$ and $\fr e$.
(A similar construction for $G = \SO(n,1)_0$ was studied in detail in \cite{arXiv:q-alg/9712040}.)

One motivation behind this construction is the following relation between $E$ and $G_\C$, the complexification of $G$.


\begin{prop}
The subspace $\fr{k}^0 \subset \fr{g}^*$ is commutative.
\end{prop}

\begin{proof}
Pick a compact form $\fr{g}_c$ of $\fr{g}_\C$ that contains $\fr{k}$.

On the one hand, the above choice of positive roots defines a Lie bialgebra structure on $\fr{g}_c^*$, and by construction $\fr{g}_c^*$ is identified with $\fr{g}^*$.

On the other, we have $\fr{k} = \fr{g}_c^\nu$ for some invariant automorphism $\nu$ of $\fr{g}_c$.
The automorphism $\nu$ must be of the form $\Ad_{\exp(x)}$ for some $x \in \s{Z}(\fr{k})$, and hence $\nu^t$ is also a Lie algebra automorphism of $\fr{g}_c^*$.
It is clear that $\fr{k}^0$ must be the $(-1)$-eigenspace of $\nu^t$ acting on $\fr{g}_c^*$, and it is also an ideal of $\fr{g}_c^*$ as $K$ is a Poisson--Lie subgroup of $G$.
Finally then,
\[
\left[ \fr{k}^0 , \fr{k}^0 \right] \subset \fr{k}^0 \cap \left( \fr{g}_c^* \right)^{\nu^t} = \{ 0 \} .
\]
This shows the claim.
\end{proof}

\begin{cor}
The group $E = \fr k^0 \rtimes K$ is a subgroup of $G_\C$.
\end{cor}

Now, take the Cartan decomposition $\fr g = \fr k \oplus \fr p$.
Then we have $\fr e \simeq \fr p \rtimes \fr k$, where we treat $\fr p$ as a vector space with an action of $\fr k$.
To recover $\fr g$ from $\fr e$, we just need to use the restriction of the bracket map $\largewedge^2 \fr p \to \fr k$ as a cocycle (say $c$), and deform the bracket of $\fr e$.
Moreover, as the standard maximal compact form $\fr g_c \subset \fr g_\C$ is $\fr k \oplus i \fr p$, the deformation of $\fr e$ by the inverted cocycle $-c$ gives $\fr g_c$.

On the other hand, $\sigma(\fr k^0)$ is a subalgebra of $\fr g_\C$ which is stable under $\ad(\fr k)$ and is isomorphic to $\fr p$ as a $\fr k$-module.
Thus, the subalgebra $\fr g' = \sigma(\fr k^0) \rtimes \fr k \subset \fr g_\C$ is isomorphic to $\fr g' \simeq \fr e$.
Moreover, $(\fr g_\C, \fr g', \fr g^*)$ is a Manin triple.

Thus, we have three Lie bialgebra structures on $\fr g^*$ corresponding to $\fr g$, $\fr g_c$, and $\fr g'$.
By the above discussion, the corresponding cobrackets $\delta_{\fr g}$, $\delta_{\fr g_c}$, and $\delta_{\fr g'}$ are different by one map $c' \colon \fr g^* \to \largewedge^2 \fr g^*$ dualizing $c$.

\begin{prop}
Let $s \in \largewedge^2 \fr p$ be the element representing the complex structure of $\fr p$.
Up to the identification $\fr p \simeq \fr p^*$ by the inner product and $\fr p^* \subset \fr g^*$ from the Cartan decomposition, $s$ satisfies
\[
\frac12 [s, s] + d s = 0,
\]
where $d$ is the differential of the Gerstenhaber algebra corresponding to $\delta_{\fr g'}$.
The twist of $\delta_{\fr g'}$ by $s$ is $\delta_{\fr g}$.
\end{prop}

That is, if $(y_j)_j$ is an orthonormal basis of $\fr p$ as a real Euclidean space, we have
\[
s = \sum_j i y_j \otimes y_j,
\]
where $i y_j$ is computed using the complex structure of $\fr p$.

By the main result of \cite{MR2271019}, quantized universal algebras $U_h(\fr g^*)$ arising from these Lie bialgebra structures on $\fr g^*$ are related by $2$-cocycles.
On the other hand, by the quantum duality principle, these can be interpreted as function algebras on the quantum groups $G_h$, $(G_c)_h$, and $E_h$.
This gives a formal analogue of the functional analytic construction of $2$-cocycles by De Commer \cite{MR2793934} for the case $\fr g = \fr{su}(1, 1)$ that connects Woronowicz's $\SU_q(2)$ \cite{MR890482} to Koelink--Kusterman's $\widetilde{\SU}_q(1,1)$ \cite{MR1962042}.

\appendix
\section{Groupoid \texorpdfstring{C$^*$}{C*}-Algebras of Matched Pairs of Lie Groups}
\label{appdix:grpd-C-star-alg}

Let $(B, C)$ be a matched pair of subgroups of a Lie group $G$.
Let us be precise about our convention of the groupoid C$^*$-algebra of the groupoid $\s G = \s G_B$.
We are going to make sense of unbounded multipliers coming from elements of $\fr c$.

First, let us fix our convention of Haar system on $\s G$.
We mostly follow the convention of \cite{MR584266}.

Let $\lambda$ be a left Haar measure on the Lie group $C$, and $\Delta$ be the modular function of $C$.
We thus have
\begin{align*}
\int_C f(c c') d \lambda(c') &= \int_C f(c') d \lambda(c'),&
\int_C f(c' c) d \lambda(c') &= \Delta(c^{-1}) \int_C f(c') d \lambda(c')
\end{align*}
for all $c \in C$ and $f \in C_c(C)$.
Let us also write
\[
\Delta(y) = \left.\frac{d \Delta(e^{t y})}{d t} \right|_{t=0} \quad (y \in \fr c).
\]

On the groupoid $\s G$, to avoid the confusion with the group structure of $G$, we write $g \cdot g'$ for the product of composable arrows and $\tilde g$ for the groupoid inverse.
Thus, given $g = r(g) c = c'' s(g)$ and $g' = s(g) c'$ with $s(g), r(g) \in B$, we have $g \cdot g' = r(g) c c'$ and $\tilde g = c^{''-1} r(g) = s(g) c^{-1}$.

For a fixed $b \in B$, the set $\s G^b$ can be identified with
\[
C^{(b)} = \{ c \in C \mid b c \in C B \},
\]
which is an open subset of $C$ by our assumption.
We define the measure $\lambda^b$ on $\s G^b$ to be the restriction of $\lambda$ on $C^{(b)}$ up to this identification.
Then the invariance condition
\[
\int f(g \cdot g') d \lambda^{s(g)}(g') = \int f(g') d \lambda^{r(g)}(g') \quad (g \in \s G_B, f \in C_c(\s G^{r(g)})
\]
follows from the left invariance of $\lambda$.
By $C^*(\s G)$ and $C^*_r(\s G)$, we mean the full and reduced groupoid C$^*$-algebras associated with this Haar system.

Next let us make sense of the associated left derivation in direction of $y \in \fr c$ as an unbounded multiplier on $C^*(\s G)$ (or on $C^*_r(\s G)$).
We follow the approach of \cite{MR1325694}*{Chapters 9 and 10}.
We consider the $C_c(\s G)$-valued inner product on $C_c(\s G)$, defined as
\[
\langle f, f' \rangle (g) = f^* * f'(g) = \int \bar f(\tilde g' \cdot \tilde g) f'(\tilde g') d \lambda^{s(g)}(g').
\]

Let us fix $y \in \fr c$, and formally write
\[
(u^s * f)(g) = \Delta(e^{s y})^{1/2} f(e^{-s y} c'' b') \quad (g = c'' b').
\]
When $f$ is compactly supported, this is well-defined for small $s$ by our openness assumption of $\s G \subset G$.
Moreover, we have
\begin{equation}\label{eq:u-s-is-right-module-map}
(u^s * f) * f' = u^s * (f * f')
\end{equation}
from left invariance of $\lambda$, and
\begin{equation}\label{eq:u-s-is-unitary}
(u^s * f)^* * f' = f^* * (u^{-s} * f')
\end{equation}
from the defining property of $\Delta$, whenever both sides are well-defined.
In particular, for any $f \in C_c(\s G)$, we have
\[
\langle f, f \rangle = \langle u^s * f, u^s * f \rangle
\]
for sufficiently small $s$.

Now, we define the operator $t^0_y$ on $C^\infty_c(\s G)$ by
\[
t^0_y f = i \left.\frac{d}{d s} u^s * f\right|_{s=0} = \frac{i}{2} \Delta(y) f + i X_y f,
\]
where $X_y$ is the smooth vector field on $\s G$ whose integral curve is $c'' b' \to e^{-s y} c'' b'$.
By \eqref{eq:u-s-is-right-module-map}, this is a map of right $C_c(\s G)$-modules.
By \eqref{eq:u-s-is-unitary}, we also have
\[
\langle t^0_y f, f' \rangle - \langle f, t^0_y f' \rangle = 0.
\]
Polynomials of $t^0_{y_1}, \dots, t^0_{y_k}$ for $y_1, \dots, y_k \in \fr c$ make sense as operators on $C^\infty_c(\s G)$.
Looking at the commutators we have
\[
[t^0_{y_1}, t^0_{y_2}] = i t^0_{[y_1, y_2]}
\]
from the standard property of $X_y$ (note that $\Delta([y_1, y_2])$ always vanishes as $\Delta$ is a homomorphism).

Next, let $f$ be a bounded continuous function on $B$.
For $f' \in C_c(\s G)$, we define the left action of $f$ on $f'$, $f * f' \in C_c(\s G)$ by
\[
(f * f')(g) = f(r(g)) f'(g).
\]
We then have $\langle f * f', f'' \rangle = \langle f', \bar f * f'' \rangle$ for $f', f'' \in C_c(\s G)$.
Thus, the left action of $f$ extends to a bounded adjointable endomorphism of $A$.
When $f \in C^\infty_b(B)$ and $f' \in C^\infty_c(\s G)$, we have $f * f' \in C^\infty_c(\s G)$.
Given $y \in \fr c$, the differentiation of the action of $e^{s y}$ on $B$ defines a vector field $X'_y$ on $B$.
We then have the commutation relation
\[
t^0_y (f * f') - f * (t^0_y f') = i X'_y(f) * f' \quad (f' \in C^\infty_c(\s G),
\]
giving a realization of the algebra $U(\fr c) \ltimes C^\infty_b(B)$ inside the space of unbounded multipliers of $C^*(\s G)$.
(In general we make no claim about regularity properties of these multipliers.)

In the case of double Lie group, $G = B C$, we can interpret $C^*(\s G)$ as the full crossed product $C_0(B) \rtimes C$ for the induced action of $C$ on $B$.
In this case there is a unital $*$-homomorphism $\s M(C^*(C)) \to \s M(C^*(\s G))$, and $t^0_y$ has an extension to an unbounded selfadjoint element affiliated to $C^*(\s G)$, see \cite{MR1232646}.

\end{document}